\documentclass[12pt,a4paper]{article}
\usepackage[utf8]{inputenc}
\usepackage[T1]{fontenc}

\usepackage{amsmath,amsthm} 
\usepackage{latexsym,amssymb}
\usepackage{graphicx, epsfig}
\usepackage{mathtools}
\usepackage{array,booktabs}
\usepackage{nomencl}

\usepackage[boxed]{algorithm}

\usepackage[inline]{enumitem}

\newcommand{\mcp}{\mathcal{P}}

\DeclareMathOperator{\crn}{crn} 

\usepackage[pdftex]{color}
\definecolor{darkgreen}{rgb}{0.0,0.5,0.0}
\definecolor{darkblue}{rgb}{0.0,0.0,0.3}
\definecolor{darkmagneta}{rgb}{0.5,0.0,0.5}
\definecolor{darkred}{rgb}{0.5,0.0,0.0}
\usepackage[bookmarks]{hyperref}
\hypersetup{colorlinks,breaklinks,
            linkcolor=darkblue,urlcolor=darkblue,
            anchorcolor=darkblue,citecolor=darkblue}
\theoremstyle{plain} 
\newtheorem{thm}{Theorem}[section]

\newtheorem{con}[thm]{Conjecture}

\theoremstyle{definition}

\newtheorem{ex}[thm]{Example}

\theoremstyle{remark}
\newtheorem{rem}[thm]{Remark}

\usepackage{authblk}

\title{On the class reconstruction number of trees} %
\author[1]{Ilia Krasikov\thanks{Ilia.Krasikov@brunel.ac.uk}}%
\author[2]{Yehuda Roditty\thanks{yr1@bezeqint.net}}%
\author[3]{Bhalchandra D. Thatte\thanks{thatte@ufmg.br}} \affil[1]{Department of
  Mathematics, Brunel Univertsity, London, U.K.}%
\affil[2]{School of Computer Science, Tel Aviv University, Israel and School of
  Computer Science, The Academic College of Tel-Aviv-Yaffo, Israel}%
\affil[3]{Departamento de Matemática, Universidade Federal de Minas Gerais, Belo
  Horizonte, Brazil}%

\date{}

\begin{document}

\maketitle

\begin{abstract}
Harary and Lauri conjectured that the class reconstruction number of trees is 2, that is, each tree has two unlabelled vertex-deleted subtrees that are not both in the deck of any other tree.  We show that each tree $T$ can be reconstructed up to isomorphism given two of its unlabelled subgraphs $T-u$ and $T-v$ under the assumption that $u$ and $v$ are chosen in a particular way.  Our result does not completely resolve the conjecture of Harary and Lauri since the special property defining $u$ and $v$ cannot be recognised from the given subtrees $T-u$ and $T-v$.

\end{abstract}

\section{Introduction}
Let $G$ be a finite simple graph. A {\em card} of $G$ is an unlabelled subgraph
of $G$ obtained by deleting a vertex of $G$. The collection (multiset) of cards
of $G$ is called the {\em deck} of $G$. For the standard graph theoretic
terminology as well as terminology about graph reconstruction, we refer to Bondy
and Murty \cite{bondy.murty.2008}.

The {\em reconstruction number} of a graph $G$, denoted by $rn(G)$, is the
minimum number of graphs in the deck of $G$ required to reconstruct $G$ up to
isomorphism.  Let $\mcp$ be a class of finite simple graphs, and let $G$ be a
graph in $\mcp$. The {\em class reconstruction number} of $G$ is the minimum
number of cards in the deck of $G$ that do not all belong to the deck of any
other graph in $\mcp$. It is denoted by $\crn_{\mcp}(G)$. We write
$\crn{(\mcp)} \coloneqq \max_G\{\crn_\mcp(G)\}$, and call it the class
reconstruction number of the class $\mcp$. The above notions were introduced by
Harary and Plantholt \cite{harary.plantholt.1985}. In this note, we take $\mcp$
to be the class of trees, and write $\crn(T)$ for the class reconstruction
number of a tree $T$.

The reconstruction number of a graph was later called the {\em
  ally-reconstruction number} by Myrvold \cite{myrvold.1990} in order to
distinguish it from the {\em adversary reconstruction number}, which is the size
of the largest subset $C$ of the deck of a graph $G$ such that $C$ determines
$G$ uniquely, but no proper subset of $C$ determines $G$ uniquely. Myrvold
\cite{myrvold.1990} refers to the class reconstruction number as {\em ally-weak
  reconstruction number.}

We refer to a survey \cite{asciak.etal.2010} for open questions related to
graph reconstruction numbers. Here we mention only a few questions and results.

Harary and Lauri \cite{harary.lauri.1987} showed that the class reconstruction
number of the class of maximal planar graphs is 2. They characterised maximal
planar graphs that have class reconstruction number 1. 

It was conjectured by Harary and Plantholt \cite{harary.plantholt.1985} that all
trees with 5 or more vertices have ally-reconstruction number equal to 3. This
conjecture was proved by Myrvold \cite{myrvold.1990}. It was proved by Harary
and Lauri \cite{harary.lauri.1988} that the class reconstruction number of the
class of trees is at most 3. They made the following conjecture.

\begin{con}[Harary and Lauri \cite{harary.lauri.1988}]\label{con-hl}
  The class reconstruction number of the class of trees is 2.
\end{con}

To our knowledge there has not been much progress on the above conjecture.

\begin{ex} Figure~\ref{fig:ex1} shows trees $T_1, T_2,T_3$ which cannot be
  reconstructed from 2 arbitrarily chosen cards from the deck. We have
  $T_1-u \cong T_2-u, T_1-v \cong T_2-v$ and
  $T_1-u \cong T_3-u, T_1-w \cong T_3-w$. Here $\crn(T_1) = 1$ since $T_1$ can
  be constructed up to isomorphism from the card $T_1-z$.

\begin{figure}[H]
  \centering
  \includegraphics[scale=1]{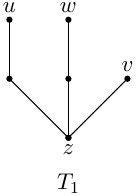}\hspace*{1cm}
  \includegraphics[scale=1]{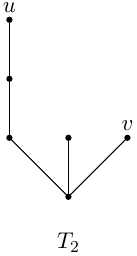}\hspace*{1cm}
  \includegraphics[scale=1]{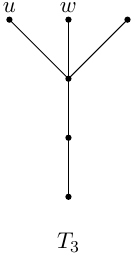}
  \caption{$T_1-u \cong T_2-u, T_1-v \cong T_2-v$ and
    $T_1-u \cong T_3-u, T_1-w \cong T_3-w$}
  \label{fig:ex1}
\end{figure}
\end{ex}

\begin{rem}
  We have $\crn(T) = 1$ if and only if there exists a vertex $u \in V(T)$ such
  that each component of $T-u$ is $K_2$ or $K_1$. This condition is equivalent
  to being `starlike' as defined by Harary and Lauri (see
  \cite{harary.lauri.1988}, Theorem 3.1).
\end{rem}

Let $T$ be a tree on at least 3 vertices. A {\em brush} in $T$ is a maximal
subgraph $H$ of $T$ that is isomorphic to $K_{1,k}$ for some $k$, such that $k$
vertices of $H$ have degree 1 in $T$, one vertex $v$ of $H$, called the {\em
  root of the brush}, has degree at least $k+1$ in $T$, and $v$ is adjacent to at
most one vertex not in $H$. A brush with $k$ edges is called a
$k$-brush.

The purpose of this note is to show the following result.

\begin{thm} \label{thm-main} Let $T,T'$ be two trees. Let $u,v\in V(T)$ and
  $u',v'\in V(T')$ such that
  \begin{enumerate}
  \item \label{c1} $u$ is a vertex of degree 1 in a brush with root $v$
  \item \label{c2} $u'$ is a vertex of degree 1 in a brush with root $v'$
  \item \label{c3} $T-u \cong T'-u'$ and $T-v\cong T'-v'$.
  \end{enumerate}
  Then $T$ and $T'$ are isomorphic.
\end{thm}

The proof of Theorem~\ref{thm-main} is based on considerations of similarity of
vertices in graphs. We define this notion in Section~\ref{sec-main}, followed by
a proof of Theorem~\ref{thm-main}. We then give an example that shows that the
above theorem doesn't settle Conjecture~\ref{con-hl}.

\section{A proof of the main theorem}\label{sec-main}

In a graph $G$, two vertices $u,v$ are called {\em similar} if there is an
automorphism of $G$ that maps $u$ to $v$.  The {\em orbit} of a vertex $u$ is
the set of all vertices similar to $u$. In a tree $T$ with set of leaves $L$,
the {\em near-leaves} are the leaves of $T-L$. We will use the following
properties of similar vertices in trees in our proof of Theorem~\ref{thm-main}.

\begin{thm}\label{thm-hp0}
  \begin{enumerate}[label=(\roman*)]
  \item (Theorem 4, \cite{hp1966}) If $T$ is a tree with leaves $a$ and $b$ such
    that $T-a\cong T-b$, then $a$ and $b$ are similar.
  \item (Corollary 3.3, \cite{kkc1983}) If $T$ is a tree with near-leaves $a$
    and $b$ such that $T-a\cong T-b$, then $a$ and $b$ are similar.
  \end{enumerate}
\end{thm}

\begin{proof}[Proof of Theorem~\ref{thm-main}]
  Let $f$ be an isomorphism from $T-u$ to $T'-u'$. Assume that $f(w) = v'$ and
  $f(v) = w'$ for some vertices $w$ and $w'$ in $T$ and $T'$, respectively. We
  have, ignoring isolated vertices in subtrees,
  \[
    (T-u)-w\cong (T'-u')-v' = T'-v' \cong T-v = (T-u)-v \cong (T'-u')-w'.
  \]
  Thus $(T-u)-w\cong (T-u)-v$, and since $v$ and $w$ are both near-leaf vertices
  or both leaf vertices of $T-u$, they are similar by Theorem~\ref{thm-hp0}.  In
  other words, there is an automorphism $g$ of $T-u$ such that $g(v) = w$. So
  the composition $fg$ is an isomorphism from $T-u$ to $T'-u'$. Now we have
  $(fg)(v) = v'$. This isomorphism extends to an isomorphism from $T$ to $T'$.
\end{proof}

The following example shows that the property that $u$ and $v$ are adjacent is
not recognisable from the partial deck $\{T-u,T-v\}$. Hence, our main theorem
does not settle Conjecture~\ref{con-hl}.

\begin{ex}\label{ex0}
  We show that in general the deck $\{T-u,T-v\}$, where $u$ and $v$ are as
  defined in Theorem~\ref{thm-main}, is not sufficient to construct the tree
  $T$. We construct non-isomorphic trees $P$ and $Q$, with vertices
  $u,v\in V(P)$ and $u',v'\in V(Q)$, such that conditions~\ref{c1} and ~\ref{c3}
  are satisfied, and the vertex $u'$ is a vertex in a brush in $T'$, but $u'$
  and $v'$ are not adjacent.

  Let $T$ be a tree rooted at $t$ with $u$ a radial vertex (i.e., an end vertex
  that lies on a longest path in $T$) in a brush and $v$ the vertex adjacent to
  it. We make 4 copies $T_{i}, i = 1, \ldots, 4$ of $T$, with corresponding
  vertices $t_i,u_i, v_i$. Moreover, assume that $T, T-u, T-v$ have the same
  height. Let $P$ and $Q$ be two trees as shows in Figure~\ref{fig-trees-pq}
  \begin{figure}[H]
    \centering \includegraphics[scale=0.85]{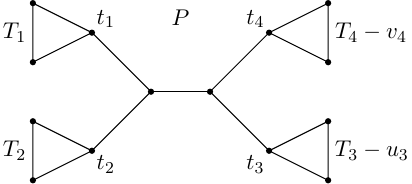}\hspace*{1cm}
    \includegraphics[scale=0.85]{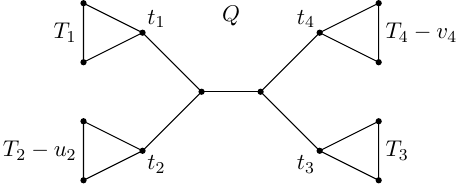}
    \caption{Trees $P$ and $Q$}
    \label{fig-trees-pq}
  \end{figure}
  Let $(u,v) = (u_1,v_1)$ and $(u',v') = (u_3,v_1)$. We have $P-u_1 \cong Q-u_3$,
  and $P-v_1 \cong Q - v_1$. But $P \not\cong Q$.
\end{ex}

It would be interesting to characterise trees for which the condition on $u$ and
$v$ can be recognised from the given cards.

\section{Acknowledgements}
\label{sec:acknowledgements}
The third author would like to thank FAPEMIG (Fundação de Amparo à Pesquisa do
Estado de Minas Gerais, Brasil) (2023-2025), Process no.  APQ-02018-22, for
financial support.


\providecommand{\bysame}{\leavevmode\hbox to3em{\hrulefill}\thinspace}
\providecommand{\MR}{\relax\ifhmode\unskip\space\fi MR }
\providecommand{\MRhref}[2]{%
  \href{http://www.ams.org/mathscinet-getitem?mr=#1}{#2}
}
\providecommand{\href}[2]{#2}

\end{document}